\documentclass[11pt,fleqn]{amsart}

\usepackage{epsfig}
\usepackage[leqno]{amsmath}
\usepackage{amssymb}
\usepackage{amscd}
\usepackage{latexsym}
\usepackage{tabularx}
\usepackage{a4wide}
\usepackage{color}
\usepackage{enumerate}
\usepackage{subfigure}

\usepackage{url}

\usepackage{url}

\usepackage[matrix,arrow,tips,curve]{xy}
\usepackage{pb-diagram,pb-xy}
\usepackage{verbatim}
\usepackage{mathrsfs}
\usepackage{color}

\input{xy}
\xyoption{all}

\small\normalsize

\newtheorem{thm}{Theorem}[section]
\newtheorem{lemma}[thm]{Lemma}

\newtheorem{prop}[thm]{Proposition}

\newtheorem{claim}[thm]{Claim}
\newtheorem*{claim*}{Claim}
\newtheorem{cor}[thm]{Corollary}

\newtheorem{remark}[thm]{Remark}

\newtheorem{defi}[thm]{Definition}

\newcommand{\Z}{\mathbb{Z}}
\newcommand{\ZZ}{\mathbb{Z}}
\newcommand{\R}{\mathbb{R}}
\newcommand{\RR}{\mathbb{R}}

\newcommand{\CC}{\mathbb{C}}

\newcommand{\ord}{\mathrm{ord}}

\newcommand{\Div}{\operatorname{Div}}

\renewcommand{\div}{\mathrm{div}}

\DeclareMathOperator{\Zh}{Zh}
\DeclareMathOperator{\supp}{supp}
\DeclareMathOperator{\Rat}{Rat}

\DeclareMathOperator\val{val}

\title{A spectral lower bound for the divisorial gonality of metric graphs}
\author{Omid Amini}
\address{CNRS - DMA, \'Ecole Normale Sup\'erieure, Paris}
 \email{oamini@math.ens.fr}
 \author{Janne Kool}
 \address{Max Planck Institut f\"ur Mathematik, Bonn}
\email{kool79@mpim-bonn.mpg.de}

\begin{document}
\maketitle

\begin{abstract}
Let $\Gamma$ be a compact metric graph, and denote by $\Delta$ the Laplace operator on $\Gamma$ with the first 
non-trivial eigenvalue $\lambda_1$. We prove the following Yang-Li-Yau type 
inequality on divisorial gonality $\gamma_{\div}$ of $\Gamma$. There is a universal (explicit) constant $C$ such that 
\[\gamma_\div(\Gamma) \geq C \frac{\mu(\Gamma) . \ell_{\min}^{\mathrm{geo}}(\Gamma). \lambda_1(\Gamma)}{d_{\max}},\]
where the volume $\mu(\Gamma)$ is the total length of the edges in $\Gamma$, $\ell_{\min}^{\mathrm{geo}}$ is the non-zero minimum length of all the 
geodesic paths between points of $\Gamma$ of valence different from two,
and $d_{\max}$ is the largest valence of 
points of $\Gamma$. Along the way, we also establish discrete versions of the above inequality concerning finite simple graph models of $\Gamma$ and their 
spectral gaps. 
\end{abstract}

\section{Introduction}

 Let $M$ be a compact Riemann surface, equipped with a metric of constant curvature in its conformal class, and denote by 
 $\lambda_1(M)$ and $\mu(M)$ the first non-trivial eigenvalue of the 
 Laplacian and the volume of $M$, respectively. Denote by $\gamma(M)$ the gonality of $M$, which is by definition, 
 the minimum degree of a branched covering $M \rightarrow \mathbb P^1(\mathbb C)$. 
 It follows from the work of Yang-Yau~\cite{YY} (see Li-Yau~\cite{LiYau} for a refinement concerning the conformal invariant of Riemannian manifolds) that 
 for any Riemann surface $M$, the following inequality holds 
 $$\lambda_1\,\mu(M)\leq 8\pi \gamma(M).$$
 
 This result has been quite useful for applications in arithmetic geometry, 
for instance in the study of rational points of bounded degree on smooth proper curves over a number field $K$, 
see for example~\cite{Abra, EHK}. Indeed a theorem of Faltings-Frey~\cite{FF} implies that curves of large gonality have 
only finite number of points defined over finite extensions of $K$ of bounded degree, and the Yang-Li-Yau inequality 
above provides a practical lower bound on the gonality in terms of geometric invariants of 
a complexification of the curve. 

\medskip

 It is quite natural to ask for the analogue type of estimates for smooth proper curves defined over 
 other base fields, 
 e.g., over global function fields. A geometric object manageable to work with is the 
 analytification of the curve over any place of the global field. For any non-Archimedean place $\nu$ 
 of a global field $K$, 
 Berkovich analytification of the curve over the completion of the algebraic closure of $K$ with respect to $\nu$ is a 
 separated compact path-wise connected topological space, which deformation retracts to a finite metric graph, which is 
 called a skeleton of the Berkovich analytification~\cite{Ber90}. 
 This finite metric graph captures  important arithmetic and geometric aspects of the original curve, see e.g. ~\cite{Zhang, CR, Cin, Thuillier, Bak, ABBR1} for background on arithmetical and algebraic geometric properties of 
 the skeleton of Berkovich analytic curves, and applications. 

\medskip

For metric graphs there are two different notions of gonality; geometric gonality, which is formulated in terms of finite 
harmonic maps from $\Gamma$ to a metric tree $T$, and divisorial gonality, which is defined in terms of the 
divisor theory. In this paper we prove a Yang-Li-Yau type inequality for the divisorial gonality in general metric graphs. 
Since the divisorial gonality is a lower bound for geometric gonality of metric graphs, our results 
improve the previous Yang-Li-Yau type inequality of~\cite{CKK}, and provides a generalization 
to arbitrary metric graphs.   The result in~\cite{CKK} was already used there to obtain a linear lower bound in the genus for the gonality of Drinfeld modular curves, which allowed to lower bound the modularity of elliptic curves over function fields, to obtain finiteness results of rational points of bounded degree on Drinfeld modular curves, and to get uniform bounds on isogenies and torsion points of Drinfeld modules. %

The proof of our result is built on the fundamental notion of tree-decomposition in 
graph minor theory. The link between gonality of graphs and their tree-decompositions was 
conjectured in~\cite{DB}, and could be an indication that graph decompositions methods and minors could be useful for further understanding of algebraic geometry of metric graphs, and for potential applications in arithmetic geometry.

 \medskip
 
 Our theorem can be stated as follows. Let $\Gamma$ be a metric graph, and denote by $\gamma_{\div}$ the 
 divisorial gonality of $\Gamma$, which is by definition, the smallest integer $d$ such that there exists a 
 divisor of degree $d$ and rank one on $\Gamma$ (we review all the basic definitions later in this introduction).

Let $\Delta$ be the (continuous) Laplacian of $\Gamma$, and $\lambda_1$ the first non-trivial eigenvalue of $\Delta$. 
Denote by $\mu(\Gamma)$ the total length of $\Gamma$,   and by $d_{\max}$ the maximum valence of points of $\Gamma$ (which is the maximum degree of any simple graph model of $\Gamma$). For a simple graph model $G$ of $\Gamma$, 
let $\ell_{\min}(G)$ be the minimum length of edges in $G$, and define $\ell_{\min}(\Gamma)$ as the maximum of $\ell_{\min}(G)$ over all simple graph models $G$ of $\Gamma$. 

\begin{thm}\label{thm:divgon} There exists a constant $C$ such that for any compact metric graph $\Gamma$ of 
total length $\mu(\Gamma)$ with first non-trivial eigenvalue $\lambda_1(\Gamma)$ of the Laplacian $\Delta$, 
the following holds
$$\gamma_\div(\Gamma) \geq C \frac{\lambda_1(\Gamma)\ell_{\min}(\Gamma)\mu(\Gamma)}{d_{\max}}.$$
\end{thm}
Our method gives a constant $C$ which is equal to $\frac 1{1024}$, however, in order to simplify the presentation, 
we do not try to optimize the constant. 

\noindent Note also that we have $\frac 12 \ell_{\min}^{\mathrm{geo}} \leq \ell_{\min} \leq \ell_{\min}^{\mathrm{geo}}$, for $\ell_{\min}^{\mathrm{geo}}$ defined as the minimum length of all the geodesic paths 
between any two points of $\Gamma$ of valence different from two, which gives the statement in the abstract up to a change in the value of the constant. 

\medskip

The proof of our theorem goes as follows. Generalizing a work of van Dobben de Bruyn~\cite{DB}, we first prove 
a theorem which relates the divisorial gonality of a metric graph to the existence of a particular 
structure in the metric graph that we call a topological bramble. 
We then relate topological brambles to structures called strong brambles in simple graph models of $\Gamma$.
We provide a dual notion for strong brambles, which is a relaxed version of tree-decomposition of graphs, that we call 
weak tree-decompositions. 
An inequality between  the weak tree-width and the tree-width of finite graphs allows to
use a spectral lower bound on tree-width to get a spectral lower bound for 
the divisorial gonality in terms of a finite simple graph model of $\Gamma$. 
A particular choice of a finite simple graph model $G$ of $\Gamma$ allows us to finish the proof of our theorem.

\medskip

As a consequence of our methods, we get the following theorem, c.f. Section~\ref{sec:random}.

\begin{thm}\label{thm:random} 
\begin{itemize}
\item The divisorial gonality of a random Erd\"os-R\'enyi graph $G(n,p)$ is $\Theta(n)$ asymptotically almost surely in the range $p>>\frac 1n$. More generally, the divisorial gonality of any metric graph whose model is a random $G(n,p)$ is $\Theta(n)$. 
\item The divisorial gonality of a random $d$-regular graph is $\Theta(n)$ asymptotically almost surely, for $d\geq 3$.  \end{itemize}
\end{thm}

(Recall that  the notation $f=\Theta(g)$ for two functions $f$ and $g$ means the existence of constants $c_1, c_2>0$ such that $ c_1 g\leq f\leq c_2 g$.)

\medskip
 
 Finally, we like to mention two independent simultaneous works~\cite{dBG} and~\cite{DJKM}. 
 In~\cite{dBG} the authors prove the above mentioned conjecture of~\cite{DB}, 
 and they extend it to metric graphs. Namely they show that the divisorial gonality of a 
 metric graph $\Gamma$ is lower bounded by the treewidth of any simple graph model of $\Gamma$. 
 Combining their theorems and our methods, it is possible to improve the current constant in 
 Theorem~\ref{thm:divgon} by a factor two. Moreover, in~\cite{DJKM}, the authors obtain a 
 sharper estimate on the divisorial gonality of random Erd\"os-R\'eny graphs. 
 More precisely, they show that the divisorial gonality of a random $G(n,p)$ is asymptotically almost surely 
$n-o(n)$, assuming $pn \rightarrow \infty$, with an estimate they provide on the error 
term.

\medskip

In the rest of this introduction, we provide necessary definitions and background on algebraic geometry and harmonic analysis on metric graphs, and recall the concept of tree-decompositions. We also discuss some direct consequences of our main theorem.

\subsection{Algebraic geometry of metric graphs}
In this section, we provide some background on algebraic geometry of metric graphs. More details can be found in~\cite{BN07, MZ08, Bak, ABBR1, ABBR2}.

\subsubsection{Metric graphs} 
Given $n\in\ZZ_{\ge 1}$, we define $S_n\subset\CC$ to be a ``star with $n$ branches'', i.e., 
a topological space homeomorphic to the union of the convex hull in $\RR^2$ of $(0,0)$ and any point among a set of $n$ points no two of them lie
on a common line through the origin.
We also define $S_0=\{0\}$.
A finite topological graph $\Gamma$ is the topological realization of a
finite graph: $\Gamma$ is a compact one dimensional topological space
such that for any point $p\in\Gamma$, there exists a neighborhood $U_p$
of $p$ in $\Gamma$ homeomorphic to some $S_n$; moreover there are only
finitely many points $p$ with $U_p$ homeomorphic to $S_n$ with $n\ne 2$.

The unique integer $n$ such that $U_p$ is homeomorphic to $S_n$ is called
the valence of $p$ and denoted $\mathrm{val}(p)$. 
A point of valence different from $2$ is called an essential vertex of
$\Gamma$: they are of two types, $v$ with $\mathrm{val}(v) \geq 3$ which are called branching points, and $v$ for which $\mathrm{val}(v)=1$ which are called ends of $\Gamma$. The set of tangent directions at $p$ is 
$T_p(\Gamma) = \varinjlim_{U_p} \pi_0(U_p\setminus\{p\})$, where the limit
is taken over all neighborhoods of $p$ isomorphic to a star with $n$
branches.   The set $T_p(\Gamma)$ has precisely $\mathrm{val}(p)$ elements.

\medskip

A metric graph $(\Gamma,\ell)$ is a compact connected metric space, such that for every $p\in\Gamma$ there is a radius $r_p\in\R_{>0}$ such that there is a neighborhood $U_p$ around $p$ which is isometric to the star shaped domain $S(\val( p ),r_p):=\{re^{2\pi i m/\val ( p )}: 0<r<r_p,1\leq m\leq \val( p )\}\subset\mathbb C$ equipped with the path-metric.  
We will usually drop the metric $\ell$ from the notation and simply refer to  $\Gamma$ as the metric, and the corresponding topological, graph. We use the notation $T_p(\Gamma)$ to denote the set of all unit tangent vectors emanating from $p$ in $\Gamma$ (which gets identified with the unit vectors $e^{2\pi i m/\mathrm{val}( p )}$ in $\mathbb C$ under the isometry of $U_p$ with $S(\mathrm{val}( p ),  r_p)$).

\smallskip
For a function $f:\Gamma\to\mathbb C$, a point $p \in \Gamma$ and a unit tangent vector $w\in T_p(\Gamma)$, the directional derivative $d_wf(x)$ of $f$ at $p$ in the direction of $w$, which we simply call the outgoing slope of $f$ at $p$ along $w$, is defined by:
\[d_wf(x)=\lim_{t\downarrow0}\frac{f(x+tw)-f(x)}{t},\]
if the limit exists. Note that the above expression makes sense by (isometrically) identifying a small enough neighborhood $U_p$ of $p$ with a star shaped domain $S(\val( p ), r_p)$ in $\mathbb C$, and by restricting $f$ to $U_p = S(\val( p ), r_p)$.

\smallskip

Let $\Gamma$ be a metric graph. A vertex set
$V(\Gamma)$ is a finite subset of the points
of $\Gamma$ which  contains all the essential points of $\Gamma$.
An element of a fixed vertex set $V(\Gamma)$ is called a vertex of
$\Gamma$, and the closure of a connected component of
$\Gamma\setminus V(\Gamma)$ is called an edge of $\Gamma$. We denote by $E(\Gamma)$ the set of all edges of $\Gamma$ with respect to the vertex set $V(\Gamma)$. The (combinatorial) graph $G = (V(\Gamma),E(\Gamma))$ is called a model of $\Gamma$. A model $G$ of $\Gamma$ is simple if there is no loop edge or double edge in $E$.
Since $\Gamma$ is a metric graph, we can associate to each edge $e$ of
a model $G=(V,E)$ its length $\ell(e)\in\mathbb R_{>0}$. 

The model $G=(V,E)$ of a metric graph $\Gamma$ with $V$ the set of all essential points of $\Gamma$ is called the minimal model of $\Gamma$. We denote by $\ell_{\min}$ the minimum length of the edges in the minimal model of $\Gamma$. The volume $\mu(\Gamma)$ of $\Gamma$ is the sum of the edge lengths in any model $G$ of $\Gamma$. We denote by $d_{\max}$ the maximum valence of points of $\Gamma$.

\subsubsection{Divisor theory on metric graphs and divisorial gonality}
 We recall some basic definitions concerning the divisor theory of metric graphs and the notion of divisorial gonality. 
 See~\cite{BN07, MZ08} for more details. 
  
For a metric graph $\Gamma$, let $\Div(\Gamma)$ be the free abelian group on points of $\Gamma$. An element $D$ of $\Div(\Gamma)$ is called a divisor on $\Gamma$ and can uniquely be written as \[D=\sum_{v\in\Gamma}a_v(v), \text{ with }a_v\in\Z,\]
where all but finitely many $a_v$ are zero. The degree of $D$ is $\deg(D)=\sum_{v\in\Gamma}a_v$. A divisor $D$ is effective if $D(v)\geq 0$ for all $v\in \Gamma$.

The set of points $v$ for which $a_v$ is nonzero is called the support of $D$ and is denoted by $\supp(D)$. 

A rational function on $\Gamma$ is a continuous piecewise linear function on $\Gamma$ whose outgoing slopes are all integers. The set of all rational functions on $\Gamma$ is denoted by $R(\Gamma)$. The order of a rational function $f$ at a point $p$ of $\Gamma$, denoted by $\ord_p(f)$, is the sum of the outgoing slopes of $f$ along the unique tangent directions in $\Gamma$ emanating from $p$. As $f$ is piecewise linear, and $\Gamma$ is compact, the order of $f$ is zero on all but finitely many points of $\Gamma$, and one gets a map $$\div:R(\Gamma)\to\Div(\Gamma), f\mapsto\sum_p \ord_p(f)( p ).$$
A divisor in the image of $\div$ is called a principal divisor.
Two divisors, $D$ and $D'$ are called linearly equivalent, written $D\sim D'$, if they differ by a principal divisor, i.e., there is a rational function such that $D=\div(f)+D'$. The (complete) linear system $|D|$ of a divisor $D$ is defined to be the set of all effective divisors which are linearly equivalent to $D$: 
\[|D|:=\{E\in\Div(\Gamma): E\geq 0, E\sim D\}.\] 
We denote by $R(D):=\{f\in\Rat(\Gamma): D+\div(f)\geq 0\}$ the "set of all global sections of $D$". Note that $R(D)$ is closed under addition by constants and under taking maximum, i.e., for $f,g\in R(D)$ and $c\in \mathbb R$, one has $c+ f \in R(D)$ and $\max(f,g)\in R(D)$, in other words, $R(D)$ is a so called tropical semi-module.

\medskip

The rank of a divisor $D$, denoted by $r(D)$ is defined by
\[r(D):=\min_{\{E:E\geq 0,|D-E|=\emptyset\}}\deg(D)-1.\]

The divisorial gonality $\gamma_{div}(\Gamma)$ of a metric graph $\Gamma$ is defined by 
\[\gamma_{div}(\Gamma):=\min\{d:\text{ there exists a } D\in\Div(\Gamma), \text{ with } \deg(D)=d \text{ and } r(D)=1\}.\]

\medskip

\subsubsection{Reduced divisors.} A basic technical tool in the study of divisors on metric graphs is the notion of reduced divisors that we recall now. 

A closed and connected subset of $\Gamma$ is called a cut in $\Gamma$. We denote by $\partial X$, the boundary of $X$: the finite set of points of $X$ which are in the closure of the complement of $X$ in $\Gamma$. For a point $v\in\partial X$, we denote by $\deg^{out}_X(v)$ the number of tangent directions in $T_p(\Gamma)$ leaving $X$ at $v$; in other words, this is the maximum number of disjoint segments in $U_v \setminus X$ whose closures have $v$ as an endpoint, where $U_p$ is a neighborhood of $v$ in $\Gamma$. 

A boundary point $v$ of a cut $X$ is called saturated with respect to a divisor $D\in\Div(\Gamma)$ if $\deg^{out}_X(v)\leq D(v)$.   

A divisor $D$ is called reduced with respect to a fixed point $v_0 \in \Gamma$ if it satisfies the following properties:
\begin{enumerate}
 \item for all $v\not=v_0, D(v)\geq 0$,
 \item for every cut $X\subset \Gamma$ such that $v_0\not\in X$, there exist a $v\in\partial X$ which is not saturated.
\end{enumerate}

Every divisor on a metric graph is equivalent to a unique $v_0$-reduced divisor, see e.g., \cite[Theorem 2]{A}.

Note that if the rank of a divisor $D$ is  non-negative, then for any $v\in\Gamma$ the reduced divisor $D_v$ is effective.

\subsubsection{Harmonic morphisms, tropical modifications, and geometric gonality of metric graphs}
We recall some standard definitions regarding the morphisms between metric graphs and the corresponding tropical curves, 
see~\cite{ABBR1} and the references there 
for a more detailed discussion of the following definitions with several examples.

Let $\Gamma$ and $\Gamma'$ be two metric graphs, and fix vertex sets $V=V(\Gamma)$ and $V'=V(\Gamma')$ for $\Gamma$ and $\Gamma'$, respectively. Denote by $E$ and $E'$ the edge sets $E(\Gamma)$ and $E(\Gamma')$, respectively. 
Let $\phi:\Gamma \to \Gamma'$ be a  continuous map.
\begin{itemize}
\item The map $\phi$ is called a $(V,V')$-morphism of metric
  graphs if we have $\phi(V)\subset V'$, $\phi^{-1}(E')\subset E$, and
the restriction
of $\phi$ to any edge $e$ in $E$ is a dilation by some factor
$d_{e}(\phi)\in\Z_{\geq 0}$.

\item The map $\phi$ is called a morphism of metric graphs if 
there exists a vertex set $V=V(\Gamma)$ of $\Gamma$ and a 
vertex set $V'=V(\Gamma')$ of $\Gamma'$ such that
$\phi$ is a $(V,V')$-morphism of metric
  graphs.

\item The map $\phi$ is said to be finite if 
  $d_{e}(\phi)>0$ for any edge $e\in E(\Gamma)$.
\end{itemize}

The integer $d_{e}(\phi)\in\Z_{\geq 0}$ in the definition above
is called the degree of $\phi$ along $e$. Let $p\in V(\Gamma)$, let $w\in T_{p}(\Gamma)$, and
let $e\in E(\Gamma)$ be the edge of $\Gamma$ in the direction of $w$.  The directional derivative of $\phi$ in the direction $w$
 is by definition the quantity
$d_{w}(\phi) := d_{e}(\phi)$.
If we set $p' = \phi(p)$, then $\phi$ induces a map
\[ d\phi(p)~:~\big\{ w\in T_{p}(\Gamma)~:~d_{w}(\phi)\neq 0 \big\}
\rightarrow  T_p(\Gamma') \]
in the obvious way.

Let $\phi:\Gamma \to\Gamma'$ be a morphism of metric graphs,
let $p\in\Gamma$, and let $p' = \phi(p)$.  The morphism $\phi$
is harmonic at $p$ provided that, for every tangent
direction $w'\in T_{p'}(\Gamma')$, the number 
\[ d_{p}(\phi):=\sum_{\substack{w\in T_{p}(\Gamma)\\w\mapsto w'}} 
d_{w}(\phi) \]
is independent of $w'$. The number $d_{p}(\phi)$  is called the degree of $\phi$ at $p$. 

We say that $\phi$ is
harmonic if it is surjective and 
harmonic at all $p\in \Gamma$; in this case
the number $\deg(\phi) = \sum_{p\mapsto p'} d_{p}(\phi)$ is
independent of $p'\in\Gamma'$, and is called the degree of
$\phi$.

 \medskip
 
 There is  an equivalence relation between metric
graphs, that we recall now; an equivalence class for this relation is called a tropical curve.

An elementary tropical modification of a metric graph $\Gamma_0$ is a
metric graph $\Gamma = [0,\ell] \cup \Gamma_0$ obtained from $\Gamma_0$ by  
attaching a segment $[0,\ell]$ of (an arbitrary) length $\ell>0$ to $\Gamma_0$ in such a way that $0 \in [0,\ell]$ gets identified with 
a point $p \in \Gamma_0$.

A metric graph $\Gamma$ obtained from a metric
graph $\Gamma_0$ by a finite sequence of elementary tropical modifications is called a tropical 
modification of $\Gamma_0$.

\smallskip

Tropical modifications generate an equivalence relation $\sim$ on the set of metric
graphs. A tropical curve
is an equivalence class of metric graphs with respect to $\sim$. 

\smallskip

There exists a unique rational  tropical curve, which is denoted by
$\mathbb T\mathbb P^1$: it is the class of all finite metric trees (which are all equivalent under tropical modifications).

A tropical morphism of tropical curves $\phi:C\to C'$ is a harmonic morphism
of metric graphs between some metric graph representatives of $C$ and  $C'$,
considered up to tropical equivalence.

 A tropical curve $C$ is said to have  a
 (non-metric) graph $G$ as its combinatorial type if $C$ admits
 a representative 
whose underlying graph is $G$.

 A tropical curve $C$ is called $d$-gonal if there exists a
tropical morphism  $C\rightarrow \mathbb T\mathbb P^1$ of degree $d$. A metric graph $\Gamma$ has geometric gonality $d$, if the tropical curve associated to $\Gamma$ is $d$-gonal, and $d$ is the smallest integer satisfying this condition. 
The geometric gonality of a metric graph is denoted by $\gamma_{gm}(\Gamma)$.

It is easy to see that the fibers of any finite harmonic morphisms from a metric graph $\Gamma$ to a finite tree are linearly equivalent, and define a linear equivalence class of divisors on $\Gamma$ of rank at least one. It thus follows that 

$$\gamma_{gm}(\Gamma) \geq \gamma_{\mathrm{div}}(\Gamma)$$ 
for any metric graph $\Gamma$. Our Theorem~\ref{thm:divgon} thus provides a spectral lower bound for the geometric gonality of a metric graph. 
 
 \subsubsection{Specialization of divisors from curves to metric graphs}
Let $X$ be a smooth proper curve over an algebraically closed complete non-Archimedean field $K$ with a non-trivial valuation. 
Recall (c.f.~\cite{BPR}, see also~\cite{Ber90, Ducros, Ducros2, Tem}) that a semistable vertex set of the Berkovich analytic curve 
$X^{\mathrm{an}}$ is a finite subset $V$ of type-2 points 
of $X^{\mathrm{an}}$ such that $X^{\mathrm{an}} \setminus V$ is a disjoint union of open balls and (a finite number of) open annuli.  
Semistable vertex sets are in bijection with semistable models of $X$ over the valuation ring of $K$. 
To each semistable vertex sets is associated a skeleton $\Sigma(X,V)$ of the Berkovich curve 
$X^{\mathrm{an}}$, which is a finite metric graph.  These metric graphs are tropically equivalent, and thus varying the semistable vertex sets defines a tropical curve $C$ associated to $X$.

Fixing a semistable vertex $V$ for $X^{\mathrm{an}}$, one gets a deformation retraction $\tau: X^{\mathrm{an}} \rightarrow \Sigma(X,V)$. Identifying $X(K)$ with points of type 1 on $X^{\mathrm{an}}$, this induces a morphism $\tau_*: \Div(X) \rightarrow \Div(\Sigma(X,V))$ which is called the specialization map. 

\begin{remark}\rm
For curves defined over an arbitrary non-trivially valued non-Archimedean field, one can find an equivalent (more classical) definition of the specialization map without reference to the analytification in~\cite{CR, Bak, Zhang}. The advantage of the above presentation is that the analytification of the curve over the algebraic closure of the completion of the base field, takes care of the renormalization by ramification indices of (the choice of) the finite base field extension over which the original curve admits semistable reduction.
\end{remark}

Let $X$ be a smooth proper curve over $K$ and let $\Gamma$ be a metric graph associated to $X$. Baker's specialization lemma~\cite{Bak} states that for any divisor $D$ on $X$ one has $r(D) \leq r(\tau_*(D))$. (Formulated in terms of the analytification of the curve,  the statement is a consequence of the Poincar\'e-Lelong formula~\cite{BPR, Thuillier} for Berkovich curves, see~\cite{AB}.)

In particular, it follows that the gonality of a smooth proper curve $X$ over a non-Archimedean field $K$ is bounded below by the divisorial gonality of the corresponding metric graph. Applying our main theorem, we get
\begin{thm} Let $X$ be a smooth proper curve over a non-Archimedean field $K$, and let $\Gamma$ be a metric graph associated to $X$. We have 
$$\gamma(X) \geq C \frac{\mu(\Gamma) \ell_{\min}(\Gamma) \lambda_1(\Gamma)}{d_{\max}}.$$
\end{thm}
Here $C$ is the constant provided by Theorem~\ref{thm:divgon}.

\subsubsection{\it Morphisms of curves induce morphisms of tropical curves} 
Let $X$ and $X'$ be two smooth proper curves over an algebraically closed complete non-Archimedean field $K$. Consider a morphism
 $\phi`: X\rightarrow X'$, and let  
$\phi: X^{\mathrm{an}}\rightarrow X'^{\mathrm{an}}$ be the induced morphism between the Berkovich analytifications of $X$ and $X'^{\mathrm{an}}$.

The proof of the following theorem, as well as more precise statements concerning stronger skeletonized versions of some foundational results of Liu-Lorenzini~\cite{LL}, Coleman~\cite{Col}, and Liu~\cite{Liu} on simultaneous semistable reduction of curves, can be found in~\cite{ABBR1}. 

\begin{thm} Let $\phi: X \rightarrow X'$ be a finite morphism of smooth proper curves over $K$ of degree $d$. 
Let $C$ and $C'$ be the tropical curves associated to $X$ and $X'$. 
Then $\phi$ induces a tropical morphism $\phi: C \rightarrow C'$ of degree $d$. 
\end{thm}
Note that, in particular, the (algebraic) gonality of $X$ over $K$ is bounded below by the (geometric) 
gonality of any metric graph $\Sigma(X,V)$. In general the inequality $\gamma(X) \geq \gamma(\Sigma(X,V))$ can be strict (see~\cite{ABBR2} 
for an example of a genus 27 tropical curve $C$ of gonality 4 such that any $X$ over $K$ of genus 27 
with associated tropical curve $C$ has gonality at least 5).

\subsection{Harmonic analysis on metric graphs} 
We recall the definitions of the Laplacian on a metric graph, and refer to~\cite{Zhang, BR, Fab} for more details on harmonic analysis on metric graphs. 

For a metric graph $\Gamma$ with a model $G=(V,E)$, one has a Lebesgue measure on each edge which gives rise to a well-defined Lebesgue measure on $\Gamma$ denoted by $dx$. The Lebesgue measure does not depend on the choice of the model. 

The space $\Zh(\Gamma)$ is the space of all continuous functions $f:\Gamma\to\mathbb R$ where $f$ is piecewise $C^2$ and $f''(x)\in L^1(\Gamma)$. The subspace $\Zh_0(\Gamma) \subset \Zh(\Gamma)$ consists of all functions $f$ which satisfy $\int_\Gamma fdx=0$. The Laplacian $\Delta$ is the measure valued operator on $\Zh(\Gamma)$ whose value on a function $f\in\Zh(\Gamma)$ is the measure 
\[\Delta(f):=-f''(x)dx-\sum_{p\in\Gamma}\left(\sum_{w\in T_p(\Gamma)}d_wf(p)\right)\delta_p(x),\] 
where $\delta_p$ the Dirac measure at the point $p$. 

The eigenvalues of $\Delta$ form a discrete subset $\lambda_0=0<\lambda_1<\lambda_2<\dots$ of $\mathbb R_{\geq 0}$. The behavior of $\lambda_i (\Gamma)$ under the scaling of the edge lengths of a model $G$ is easily seen to be as follows: if the metric graph $\Gamma'$ with the same model $G$ is obtained from $\Gamma$ by scaling the length in $\Gamma$ of each edge $e\in E(G)$ with a factor $\beta\in\R_{>0}$, then $\lambda_i(\Gamma')=\frac{1}{\beta^2}\lambda_i(\Gamma)$~\cite{BR}.

The smallest non-zero eigenvalue of $\Delta$, $\lambda_1(\Gamma)$, has the following variational characterization
\[\lambda_1(\Gamma)=\inf_{f\in\Zh_0(\Gamma)}\frac{\int_\Gamma|f'|^2dx}{\int_\Gamma f^2dx}.\]

\subsection{Tree-decompositions}\label{intro:tw}

 Let $G=(V,E)$ be a connected graph. A
tree-decomposition of $G$ is
a pair $(T,{\mathcal X})$ where $T$ is a finite tree on a set of vertices $I$, and ${\mathcal X}=\{X_i : i\in I\}$ 
is a collection of subsets of
$V$, subject to the following three conditions: 
\begin{enumerate}
\item $V=\cup_{i\in
  I}X_i$, 
  \item for any edge $e$ in $G$, there is a set $X_i\in {\mathcal
  X}$ which contains both end-points of $e$, 
  \item for any triple
$i_1,i_2,i_3$ of nodes of $T$, if $i_2$ is on the path from $i_1$ to
$i_3$ in $T$, then $X_{i_1}\cap X_{i_3}\subseteq X_{i_2}$, or equivalently, for any vertex $v$ in $G$, the set of nodes $i$ of $T$ with $v \in S_i$ form a connected subtree of $T$.
\end{enumerate}

Note that the point (3) in the above definition 
simply means that the subgraph of $T$ induced by all the vertices $i$ which 
contain a given vertex $v$ of the graph $G$ is connected.

The width of a tree-decomposition $(T,{\mathcal X})$ is defined as $w(T,{\mathcal
  X})=\max_{i\in I}|X_i|-1$. The tree-width of $G$, denoted by $tw(G)$,  is the minimum width of any 
  tree-decomposition of $G$.  

  \medskip

 There is a useful duality theorem concerning the tree-width which allows in practice to bound the tree-width of graphs. 
 The dual notion for tree-width is bramble (as named by B. Reed~\cite{Reed}): 
 a bramble in a finite graph $G=(V,E)$ is a collection of 
 connected subsets of $V$ (i.e., those inducing a connected subgraph) such that the union of any two of these subsets form again a connected subset of $G$. 
 The order of a bramble $\mathcal F$ in $G$ is the minimum size of a hitting set for $\mathcal F$, i.e., the minimum size of a subset of vertices which has non-empty intersection with any element of $\mathcal F$.  The bramble number of 
 $G$ denoted by $bn(G)$ is the maximum order of any bramble in $G$.
 
 \begin{thm}[Seymour-Thomas~\cite{ST}]
  For any graph $G$, $tw(G) = bn(G)$.
 \end{thm}
 
 To give an example of the applications of the duality theorem,  let $H$  be an $n\times n$ grid. It is easy to see that $bn(H) =n$ by taking brambles formed by crosses in the grid. This shows that grid graphs can have large tree-width, and so the tree-width can take arbitrary large values on planar graphs.  

Duality theorems are part of Robertson-Seymour graph minor theory~\cite{RS}.  For a discussion of the different duality theorems and diverse generalizations see~\cite{AMNT, DO}.

 \section{Topological brambles and divisorial gonality}
  In this section we provide a lower bound on the divisorial gonality in terms of a topological variant of 
  the notion of bramble in finite graphs. 
  \begin{defi}[Topological bramble]\label{top-bramble}\rm
   Let $\Gamma$ be a metric graph. A topological bramble (or simply top-bramble) in $\Gamma$ is a finite family
   $\mathcal F$ of non-empty 
   closed connected metric subgraphs of $\Gamma$ such that any two elements $X$ and $Y$ in $\mathcal F$ have a non-empty intersection.
   The order of a top-bramble $\mathcal F$ is the minimum size of a hitting set for $\mathcal F$, i.e., the minimum size of a set 
   $S \subset \Gamma$ such that $S \cap X \neq \emptyset$ for any $X \in \mathcal F$. The topological bramble number of 
   $\Gamma$ denoted by 
   $tbn(\Gamma)$ is the maximum order of any topological bramble on $\Gamma$.
  \end{defi}

 \begin{thm}\label{thm:divgontopbram}
  Let $\Gamma$ be a metric graph. The divisorial gonality of $\Gamma$ is lower bounded by 
  its top-bramble number $tbn(\Gamma)$.
 \end{thm}
\begin{remark}\rm
  In the next section, we will provide a link between topological brambles and 
  a special kind of brambles in finite simple graph models of $\Gamma$, called strong brambles. In view of this link, this
  theorem can be seen as a generalization of a theorem of van Dobben de Bruyn~\cite{DB} to metric graphs.
\end{remark}

\begin{proof}
  In order to prove the theorem, we need to show that there cannot exist any divisor of degree $k$ 
  and rank at least one in $\Gamma$ provided that there exists a top-bramble $\mathcal F$ in $\Gamma$ of order $k+1$. For the sake of a contradiction, let $\mathcal F$ be a top-bramble of order $k+1$ 
  for $\Gamma$ and assume there exists a divisor $D$ with $\deg(D) = k$ and $r(D)\geq 1$. 
  In particular, for any point $x$ of $\Gamma$, the reduced divisor $D_x$ has $x$ in its support. 
  
  Consider the linear system $|D|$ of $D$.
  Without loss of generality, and by replacing $D$ by another divisor $E \in |D|$ if necessary, 
  we can assume that $D$ is effective, and in addition, that $D$ is the divisor in the linear system 
  $|D|$ whose support $\mathrm{supp}(D)$ 
  has the maximum number of non-empty intersections $\mathrm{supp}(D) \cap X$ with elements $X$ in $\mathcal F$. 
  
  Since $\mathrm{ord}(\mathcal F)= k+1 >  |\mathrm{supp}(D)|$, there exists an element 
  $X$ in $\mathcal F$ such that $X \cap \,\mathrm{supp}(D) = \emptyset$. Let $v$ be an arbitrary
  point of $X$, and consider the unique $v$-reduced divisor $D_v \sim D$. Let $f$ be a rational function on 
  $\Gamma$ which gives $\div(f) + D= D_v$. Since $r(D) \geq 1$, we have $v\in \mathrm{supp}(D_v)$, while by the choice of $X$, 
  we have $v\notin \mathrm{supp}(D)$.
  
  The structure of the proof is as follows: we consider a specific path $D_t$ in $|D|$ (parameterized by $t$) from $D_v$ to $D$. Since $\mathrm{supp}(D_v) \cap X \neq \emptyset$ while $\mathrm{supp}(D) \cap X = \emptyset$, by compactness of $X$, there exists a maximum value $h$ of $t$ such that $\mathrm{supp}(D_h) \cap X \neq \emptyset$. Denote by $X_1, \dots, X_n$ all the different elements of $\mathcal F$ which have non-empty intersections with $\mathrm{supp}(D)$. Note that the support of   $D_h$ intersects $X$, and $X$ is not among the $X_i$'s, so by the choice of $D$, as the one maximizing the number of non-empty intersections with elements in $\mathcal F$, there should exist an element $Y=X_i$ such that $Y \cap \mathrm{supp}(D_h) = \emptyset$. We will show that $Y\cap X=\emptyset$, which is in contradiction with the definition of a topological bramble.   
  
  \medskip
  
   For any real number $t$, define a function $f_t$ on $\Gamma$ by $f_t(x):= \max(f(x),t)$ for any $x\in \Gamma$. Since both $f$, and the constant function $t$ lie in $R(D)$ and $R(D)$ is a tropical semi-module it follows that $f_t\in R(D)$ for all $t$. In other words, the divisor $D_t:=D+\div(f_t)$ is in $|D|$. Denote now by $\min f$ and $\max f$ the minimum and maximum value of $f$ on $\Gamma$ respectively. The map $[\min f,\max f]\to |D|, t\to D+\div(f_t)$ defines a path in $|D|$ from $D_v$ to $D$, since $f_{\min f}=f$ and $f_{\max f}$ is the constant function $\max f$. Next, we prove the following claim.
   
  \begin{claim*}[1]For any real number $t$ denote by $f^{-1}(t) = \{x \in \Gamma\,|\, f(x) =t\}$ the level set at $t$. We have $\partial f^{-1}(t) \subseteq \mathrm{supp}(D_t)$.
  \end{claim*}
  
 \noindent \emph{Proof of Claim (1): } Define the upper level set $\Gamma_t := \{x\in \Gamma\,|\, f(x) \geq t\}$. 
 First note that writing $f = (f-f_t) + f_t$, we have $D_v = D+\div(f) = \div(f-f_t) + \div(f_t)+D = 
 \div(f-f_t) +D_t$, which shows that $f-f_t \in R(D_t)$. 
 Note also that  $f-f_t$ is constant on $\Gamma_t$ and coincides with $f-t$ outside $\Gamma_t$. Consider now a point $x\in \partial f^{-1}(t)$, in other words, $f$ does not restrict to a constant function on any neighborhood of $x$. If $x \in \partial \Gamma_t$ since $f$ (and so $f-f_t$) is strictly decreasing along any out-going branch $e$ 
 from $\Gamma_t$ at $x$, the slope of $f-f_t$ at $x$ along $e$ is strictly negative. Since $f-f_t \in R(D_t)$, this shows that 
  $x \in \mathrm{supp}(D_t)$, and the claim follows. If $x \not \in \partial \Gamma_t$, since $f$ is not a constant function locally at $x$, 
  then $f$ is strictly increasing along one of the branches at $x$, and so again, since $f_t$ takes its minimum value at $x$, $D$ is effective, and
  $D_t =D + \div(f_t)$, we conclude that $x\in \mathrm{supp}(D_t)$, and the claim follows.  
  
  \medskip
    Consider now  
  the path in $|D|$ defined by all the divisors $D_t = D + \div(f_t)$, for the values of $t\in [\min f, \max f]$.  Let $h$ be the maximum value in $ [\min f, \max f]$ such that  $\mathrm{supp}(D_h) \cap X \neq \emptyset$.  
  
  \begin{claim*}[2]
 We have  $X\cap\Gamma_h\subset f^{-1}(h)$.  
  \end{claim*}
  
\noindent  \emph{Proof of Claim (2):}  
 Since $-f\in R(D_v)$ it follows from~\cite[Lemma 7]{A} that the minimum of $f$ is taken at $v$. Hence, any path $P$ in $X$ from $v$ to a point in $f^{-1}(\max f_{|X}) \cap X$ intersects $\partial f^{-1}(\max f_{|X})$, which combined
 with Claim $(1)$, implies that $\mathrm{supp}(D_{\max f_{|X}})$ intersects $X$. On the other hand, for any $t > \max f_{|X}$, since $\mathrm{supp}(D_t) \subseteq \Gamma_t$, and $\Gamma_t\cap X = \emptyset$, the intersection
  $\mathrm{supp}(D_t) \cap X$ is empty. The claim follows.

 \medskip
 
Let now  $Y$ be an element of $\mathcal F$ with the property that $Y \cap \mathrm{supp}(D_h) = \emptyset$ while $Y \cap \mathrm{supp}(D)\neq \emptyset$. The following claim will imply that $Y\cap X=\emptyset$, which is in contradiction with the definition of a topological bramble,   and our theorem follows.

 \begin{claim*}[3]
 We have $Y \subseteq \Gamma_h\setminus f^{-1}(h)$, in other words, $f_{|Y} > h$.   
 \end{claim*}

  We make the following observation which will be used in the proof of the above claim. 
    
    \medskip
    
 \noindent \emph{Observation.}   Let $y$ be a point in the intersection $Y \cap\, \mathrm{supp}(D)$ (which is by assumption non-empty). Since $y \not \in \mathrm{supp}(D_h)$, the point $y$ is not a local minimum of $f_h$, i.e., all the slopes of $f_h$ along the adjacent 
 branches at $y$ cannot be non-negative. In particular, $Y$ cannot be entirely contained in $f^{-1}(h)$.
 \medskip

\noindent \emph{Proof of Claim (3):} First, note that $Y \cap \Gamma_h \neq \emptyset$: indeed, otherwise, this would mean that the restriction 
  $f|_Y < h$, and so any point of $Y$ would be a local minimum for $f_h$, contradicting the above observation.
 Second, we show that $Y \subset \Gamma_h$, i.e., $\min f_{|Y} \geq h$. Otherwise, there would exist a point $z \in Y$ 
 such that $f(z) <h$. By connectivity of $Y$, and by taking a path $P \subset Y$ from $z$ to a point in 
 $\Gamma_h \cap Y$, we would obtain that $\partial f^{-1}(h) \cap Y \neq \emptyset$, which by Claim (2)
 would lead to $\mathrm{supp}(D_h) \cap Y \neq \emptyset$, contradicting the choice of $Y$. 
 Combining this with the above observation, since $Y$ is not entirely contained in $f^{-1}(h)$, we infer that $Y \cap (\Gamma_h \setminus f^{-1}(h)) \neq \emptyset$. We finish the proof of Claim (3) by showing that $Y\cap f^{-1}(h) = \emptyset$: suppose there is an element in $Y\cap f^{-1}(h)$, then, by connectivity of $Y$, there would exist a path $P$ in $Y$ from that element to a point in $Y \cap (\Gamma_h \setminus f^{-1}(h))$. This path would contain a point in $\partial f^{-1}(h)$, which by Claim (2) would imply that  $\mathrm{supp}(D_h) \cap Y \neq \emptyset$, contradicting again the choice of $Y$. This finishes the proof of the claim, and our theorem follows.
 \end{proof}

\section{Topological brambles in metric graphs vs strong brambles in graphs}
Let $G=(V,E)$ be a finite simple graph on vertex set $V$ and with edge set $E$. 
 We gave the definition of a bramble in Section~\ref{intro:tw}, and mentioned that it provides a dual notion for 
 tree-width. A strong bramble is a specific kind of bramble in a finite graph defined as follows:
 \begin{defi}[Strong bramble]\rm
 A strong bramble in $G$ is a finite collection $\mathcal F$ of connected subsets of $G$ such that for any two elements 
 $B$ and $C$, one has $B \cap C \neq \emptyset$. 
 \end{defi}

 Note in particular that for any two elements $B$ and $C$ in $\mathcal F$, the union $B \cup C$ is obviously connected. 
 In other words, a strong bramble is a bramble. 
 
 The order of a strong bramble is its order as a bramble, i.e., the minimum size of a hitting set for $\mathcal F$ in $V$. The strong bramble number of a finite graph $G$, 
 denoted by $sbn(G)$, is the maximum order of any strong bramble in $G$.
 
 \medskip
 
 Let $\Gamma$ be a metric graph with a simple graph model $G=(V,E)$. For any subset $X$ of $V$, we denote by $G[X]$, 
 resp. $\Gamma[X]$, the subgraph of $G$, resp. the 
  metric subgraph of $\Gamma$, defined by $X$: it contains all the edges of $G$, resp. all the metric edges of $\Gamma$, 
  which connect a vertex in $X$ to another vertex of $X$.

  \medskip

The link between strong and topological brambles is provided in the following proposition.
  \begin{prop}\label{prop:link}
   \begin{itemize}
    \item[(1)] Let $\Gamma$ be a metric graph with a simple graph model $G=(V,E)$. For any strong bramble 
    $\mathcal F$ in 
    $G$, the collection $\Gamma[\mathcal F] = \{\Gamma[X]\, | \, X\in \mathcal F\}$ is a topological 
    bramble for $\Gamma$ of the same order.
    \item[(2)] For any topological bramble $\mathcal F_\Gamma$ for $\Gamma$, there exists a simple graph model $G$ 
    of $\Gamma$ and a strong bramble $\mathcal F$ for $G$ such that 
    $\mathcal F_\Gamma = \mathcal F[\Gamma]$.
    \item[(3)] The topological bramble number of $\Gamma$ is given by $\sup_{G} sbn(G)$ where the supremum is over 
    all simple graph models $G$ of $\Gamma$ and $sbn(G)$ is the strong bramble number of $G$.
   \end{itemize}
In addition, the topological bramble number of any metric graph is finite, in other words, the supremum is a maximum. 
  \end{prop}

  \begin{proof}
   (1) Let $\mathcal F$ be a strong bramble for a simple graph model $G$ of $\Gamma$. 
   Obviously, for any $X\in \mathcal F$, the subset $\Gamma[X]$ is a connected metric subgraph of $\Gamma$. 
   In addition, for any two elements $X, Y \in \mathcal F$, we have 
   $\Gamma[X]\cap \Gamma[Y] \supset X \cap Y \neq \emptyset$
   which shows that $\Gamma[\mathcal F]$ is a topological bramble for $\Gamma$. To see that 
   $\ord(\Gamma[\mathcal F]) = \ord(\mathcal F)$, note first 
   that since any hitting set for $\mathcal F$ is also a hitting set for 
   $\Gamma[\mathcal F]$, we obviously have $\mathrm{ord}(\Gamma[\mathcal F]) \leq \mathrm{ord}(\mathcal F)$. 
   To prove the equality, it will be clearly enough to show the existence of a minimum size hitting set 
   $S$ for $\Gamma[\mathcal F]$ such that $S \subset V(G)$ ($S$ will be then a hitting set for $\mathcal F$). 
   Let $S$ be a hitting set of minimum size for $\Gamma[F]$. For any point $x$ in $S \setminus V$ which lies in the 
   interior of an edge $e_x$ of $G$, chose one of the two extremities $v_x$ of $e_x$ and replace $x$ with $v_x$ 
   to obtain a set $\tilde S$ of the same size $|S|$. Note that 
   a metric subgraph of the form $\Gamma[X]$ in $\Gamma[\mathcal F]$ which contains the point $x \in S \setminus V$ 
   contains both the end-points of the edge $e_x$, and has non-empty intersection with $\tilde S$. An element of 
   $\Gamma[\mathcal F]$ which intersects $S \cap V$ has also non-empty intersection with $\tilde S$. It follows 
   that $\tilde S$ is a hitting set for 
   $\Gamma[\mathcal F]$, and the claim follows. 
 \medskip
 
 \noindent (2) Let $\mathcal F_\Gamma$ be a topological bramble for $\Gamma$, and let $G_0 = (V_0, E_0)$ be 
 a simple graph model for $\Gamma$. Consider the set $\cup_{X \in \mathcal F_\Gamma}\partial X$ of all points of $\Gamma$ which lie on the
 boundary of a set $X \in  \mathcal F_\Gamma$ and define $V_1 =  V_0 \cup \cup_{X \in \mathcal F_\Gamma}\partial X$. Consider the model $G_1 = (V_1,E_1)$ of $\Gamma$ defined by $V_1$. Finally, 
 subdivide each edge $e$ of $G_1$ by adding a new vertex in the middle of $e$ to obtain a model $G = (V, E)$ of $\Gamma$. 
 Let $\mathcal F = \{V \cap X\,|\, X \in \mathcal F_\Gamma\}$. It is not hard to see that $\mathcal F$ is a strong bramble and $\Gamma[\mathcal F] = \mathcal F_\Gamma$.  
 
 \medskip
 
 \noindent (3) The equality of the topological bramble number and the supremum $\sup_{G} sbn(G)$, 
 for $G$ a model of $\Gamma$, formally follows from the two assertions 
 (1) and (2). Finiteness of the topological bramble number of a metric graph is  a consequence of 
 Theorem~\ref{thm:divgontopbram}, since by Brill-Noether bound (or simply Riemann-Roch) for metric graphs, 
 the divisorial gonality of any metric graph is finite.
 \end{proof}

 \section{Strong brambles and weak tree-decompositions}
 In this section, we provide the dual notion to strong brambles: we introduce a new class of graph 
 decompositions that we call 
 weak tree-decompositions. Once this has been discovered, it is not difficult to show that 
 strong brambles of given order are the dual obstructions for the existence of weak tree-decompositions of given order, 
 see Theorem~\ref{thm:stbramwt} 
 below for a precise formulation.

 \medskip
 
  Let $G=(V,E)$ be a connected graph. A
weak tree-decomposition of $G$ is
a pair $(T,{\mathcal S})$ where $T$ is a finite tree on a set of nodes $I$, and ${\mathcal S}=\{S_i : i\in I\}$ 
is a collection of subsets of
$V$, subject to the following three conditions: 
\begin{enumerate}
\item $\cup_{i\in
  I}S_i = V$, 
  \item for any edge $e$ in $G$, there is an edge $\{i,j\}$ in $T$ such that  $ e\subset S_i \cup S_j$.
  \item for any vertex $v$ in $G$, the set of nodes $i$ of $T$ with $v \in S_i$ form a connected subtree of $T$.
\end{enumerate}
For any vertex $v \in V$, we denote by $T_v$ the (connected) subtree of 
 $T$ which is induced by all the nodes $i$ of $T$ with $v \in S_i$.
 
Note that the only difference with the usual definition of a tree-decomposition is in point (2) where we impose a weaker
 condition. In particular, it might happen that an edge $e$ of $G$ is not necessarily contained in any set 
 $S_i \in \mathcal X$.

The width of a weak tree-decomposition $(T,{\mathcal S})$ is defined as $w(T,{\mathcal
  S})=\max_{i\in I}|S_i|$. The weak tree-width of $G$, denoted by $wtw(G)$,  is the minimum width of any weak 
  tree-decomposition of $G$.

 \begin{lemma}\label{lem:con} Let $(T, \mathcal S)$ be a weak tree-decomposition of a graph $G=(V,E)$. 
 For any two adjacent vertices $u$ and $v$, 
 $T_u \cup T_v$ is connected. In particular, for any connected subset $X$ of $G$, the union $\cup_{v\in X} T_v$ 
 is a connected subtree of $T$.
 \end{lemma}
 \begin{proof} By property (2) of a weak tree-decomposition, the edge $\{u,v\}$ is contained in the union of 
 two sets $S_i$ and $S_j$ in $\mathcal S$ for an edge  
 $\{i,j\}$ of $T$. This means that  either $T_u \cap T_v \neq \emptyset$ or $i$ and $j$ do not belong to the same tree 
 among $T_u$ and $T_v$. In any of the two cases, $T_u \cup T_v$ is connected. 
 The second statement obviously follows by connectivity of $X$ and the first assertion.
 \end{proof}

The following proposition is straightforward from the definition.
\begin{prop}\label{prop:res}
\begin{itemize}
 \item Let $(T, \mathcal S)$ be a weak tree-decomposition of a graph $G=(V,E)$, and let 
$U$ be a subset of $V$. 
The restriction of $(T, \mathcal S)$ to $U$ defined by replacing any $S_i$ in the decomposition with $S_i\cap U$ is a weak tree-decomposition 
of $G[U]$.

\item Let $(T, \mathcal S)$ be a weak tree-decomposition of a graph $G=(V,E)$, and $i$ and $t$ two nodes of $T$. 
Let $v\in S_t$, and for any node $j$ on the unique path between $i$ and $t$, define $S'_j =S_j \cup \{v\}$. 
For all the other nodes $j$ of $T$, define $S'_j = S_j$. Then $(T, \mathcal S'=\{S'_j\})$ is a weak tree-decomposition of $G$.
\end{itemize}
 \end{prop}
 \begin{proof}
  The first assertion is obvious from the definition. For the second one, we only need to verify the property (3) for the vertex $v$. The tree 
  $T'_v$ associated to $(T, \mathcal S')$ is the union of $T_v$ and the unique path in $T$ between $i$ and $t$. Since $t$ also belongs to $T_v$, $T'_v$ is connected. 
 \end{proof}

 The following theorem is a duality theorem, in the spirit of the duality theorems in graph minor theory, 
 which relates strong brambles to weak tree-decompositions. It does not seem to follow from the generalized forms of duality established in~\cite{AMNT, DO}, so we provide a proof. 
 \begin{thm}\label{thm:stbramwt} A finite graph $G$ has a weak tree-decomposition of width $k$ 
 if and only if there is no strong bramble of order strictly larger than $k$ in $G$. In other words, $wtw(G) = sbn(G)$.
 \end{thm}
The proof mimics the well-known proof of the duality theorem between tree-width and bramble order~\cite{Diestel, ST},  and is based on the use of Menger's theorem in graph theory. 
 We thus start by recalling the statement of Menger's theorem. 
 
 Let $G=(V,E)$ be a finite simple graph and consider two subsets $X, Y \subseteq V$.
An $(X,Y)$-separator in $G$ is a subset $S \subset V$ such that there is no path in $G \setminus S$ between any point in $X \setminus S$ and $Y \setminus S$.

The connectivity of the pair $(X,Y)$ is the maximum number of vertex disjoint path between $X$ and $Y$ (a point $v\in X \cap Y$ is considered as a path between $X$ and $Y$ of length zero).

Consider a set of $k$ vertex disjoint paths between $X$ and $Y$. Obviously any $(X,Y)$ separator in $G$ 
should contain at least one point on each of the $k$ paths. In other words, the size of any $(X,Y)$ separator is at 
least the connectivity of the pair $(X,Y)$ in $G$. Menger's theorem asserts the equality of these two quantities.

\begin{thm}[Menger~\cite{Menger}]
The connectivity of a pair  $(X,Y)$, $X, Y \subseteq V$, is equal to the minimum size of an $(X,Y)$-separator in $G$. 
\end{thm}

We are now ready to prove Theorem~\ref{thm:stbramwt}.
\begin{proof}[Proof of Theorem~\ref{thm:stbramwt}] We first show that $sbn(G) \leq wtw(G)$.
 Let $(T,\mathcal S)$ be a weak tree-decomposition of $G$. Consider a strong bramble $\mathcal F$ for $G$. 
 We show that there exists a node $i$ such that 
 $S_i$ intersects any element $X \in \mathcal F$, i.e., $S_i$ is a hitting set for $\mathcal F$. This proves the claimed inequality. 
 For the sake of a contradiction, suppose this is not the case. 
 This means that 
 for any node $i$ in $T$,  there exists $X_i \in \mathcal F$ such that $X_i \cap S_i =\emptyset$, 
 in other words, $T_v$ does not contain node $i$ for any $v \in X_i$. 
 Thus, the union $T_i = \cup_{v\in X_i} T_v$ does not contain $i$. In addition,  
 by Lemma~\ref{lem:con}, $T_i$ is a connected subtree of $T$.  
 This implies that $T_i$ is entirely included in one of the connected components of $T \setminus i$. 
 Let $j$ be the unique node of this connected component which is adjacent to $i$, and give the orientation $ij$ to the 
 edge $\{i,j\}$ of $T$. Doing this for any node $i$ of $T$, we give the size of $V(T)$ orientations to the edges of $T$. 
 Since $T$ contains $|V(T)|-1$ edges, there exists an edge $\{i,j\}$ which gets both the orientations $ij$ and $ji$.
 This precisely means that the two trees $T_i$ and $T_j$ are disjoint, which implies that 
 $X_i \cap X_j =\emptyset$, contradicting the defining property of a strong bramble.

 \medskip
 To prove $wtw(G) \leq sbn(G)$, we show the existence of a weak tree-decomposition of $G$ of order at most $k:=sbn(G)$.
 The proof goes by an induction procedure as follows.  We claim that for any graph $G$ and for any strong bramble 
 $\mathcal F$ in $G$, there exists a weak tree-decomposition $(T,\mathcal X)$ such that for any node 
 $i$ of $T$, either 
 $|X_i| \leq k $ or, otherwise if $|X_i| > k$, then $X_i$ is not a hitting set for $\mathcal F$. 
 The proof of this latter statement is by a reverse 
 induction on  $|\mathcal F|$ for any strong bramble $\mathcal F$ in $G$. 
 Once this has been proved, for the empty strong bramble $\mathcal F = \emptyset$, since every set is a hitting set for $\mathcal F$, 
 we get a weak  tree-decomposition of $G$ of width at most $k$, and the theorem follows. 
 
 Since there is no strong bramble of size larger than $2^{|G|}$ in $G$, the base of the induction 
 holds trivially for strong brambles of size $2^{|G|}+1$ (which do not exist). Suppose that the statement is true 
 for an integer $N$ and any strong bramble 
 $\mathcal F$ of size $|\mathcal F| =N$ in $G$, we show that it also holds for all strong brambles of size $N-1$.
 
For the sake of a contradiction, suppose that the statement does not hold for $N-1$. 
Let $\mathcal F$ be a strong bramble with $|\mathcal F| = N-1$ for which the statement does not hold. 
 Consider a hitting set $S$ of  $\mathcal F$ in $G$ of size equal to the order of $\mathcal F$. Note that $|S| \leq k$.
 
Let $C_1, \dots, C_l$ be all the connected components of $G \setminus S$, and for any integer $1\leq a\leq l$, 
 consider the induced subgraph $G^a$ of $G$ with vertex set 
 $S \cup V(C_a)$. We will show that there exists a weak tree-decomposition $(T^a, \mathcal S^a = \{S^a_j\}_{j\in V(T^a)})$ 
 of $G^a$ such that
 \begin{itemize}
  \item[$(i)$] There is a node $i_a$ in $T^a$ such that $S^a_{i_a} = S$;
  \item[$(ii)$] For any node $j$ of $T^a$ with $|S^a_j| >k$, $S^a_j$ is not a hitting set for $\mathcal F$.
 \end{itemize}
Once this statement is proved, we obtain a weak tree-decomposition of the whole graph by gluing the trees 
$T^a$ on the node $i_a$ to obtain a tree 
$T$, and by defining for any node $j$ of the tree $T$, which is thus a node of one of the trees $T^a$ for some $a$, 
$S_j :=S^a_j$. 
Since $S^a_{i_a} = S$, these sets are well-defined, and it is easy to verify that they form a weak tree-decomposition 
of the whole graph $G$. 
In addition, by Property $(ii)$ above, any $S_j$ of size strictly larger than 
$k$ belongs is not a hitting set for $\mathcal F$. 
Thus, $(T,\mathcal S)$ is a 
weak tree-decomposition for $G$ which 
satisfies the required property with respect to $\mathcal F$, and this leads to a contradiction with our choice of 
$\mathcal F$.

We are thus left to prove the above claim. Consider one of the graphs $G_a$. There are two cases to consider:

\begin{enumerate}
 \item Either $C_a$ is not a hitting set for $\mathcal F$.
\end{enumerate}

In this case, we get a weak tree-decomposition $(T^a, \mathcal S^a)$ of $G^a$ by taking a path of length 
two $T^a$ on two vertices $i_a$ and $j$, and by defining 
$S^a_{i_a} =S$  and $S^a_j = C_a$. Obviously $(i)$ and $(ii)$ are verified. 

\begin{enumerate}
  \item[(2)] Or $C_a$ is a hitting set for $\mathcal F$.
\end{enumerate}

Since $C_a$ is connected, this precisely means that $\mathcal F_a =\mathcal F \cup \{C_a\}$ is a strong bramble in $G$. 
Note that $S$ is a hitting set for $\mathcal F$ and $S \cap C_a = \emptyset$, so $\mathcal F_a \neq \mathcal F$, and thus 
$|\mathcal F_a| = |\mathcal F|+1 =N$. By the hypothesis of our induction, there exists a weak tree-decomposition 
$(T, \mathcal S)$ for $G$ such that any $S_i$, for a node $i$ in $T$, with $|S_i|>k$ is not a hitting set 
for $\mathcal F_a$. By the choice of $\mathcal F$,  the weak tree-decomposition $(T, \mathcal S)$ has a node  $i_a$ in
$T$ with $|S_{i_a}| >k$ such that $S_{i_a}$ is a hitting set for $\mathcal F$. 
We must have $S_{i_a} \cap C_a =\emptyset$ (otherwise, $S_{i_a}$ 
would be a hitting set for $\mathcal F_a$).

\medskip
We would like to restrict this weak tree-decomposition to $G^a$, which by Lemma~\ref{prop:res}, is 
a weak tree-decomposition of $G^a$. However, the restriction does not necessarily verify properties 
$(i)$ and $(ii)$, in particular, $S_{i_a} \neq S$, so we use Menger's theorem in order to slightly modify the 
restriction of $(T, \mathcal S)$ to $G^a$, making it satisfy $(i)$ and $(ii)$.

By applying Menger's theorem, we first show  
that there are $|S|$ vertex disjoint paths from $S$ to $S_{i_a}$ in $G$. Consider thus
an $(S, S_{i_a})$-separator $A$ in $G$. We are reduced to proving that $|A| \geq |S|$. Suppose this is not the case and so $|A| < |S|$.  
Since the order of $\mathcal F$ is equal to $|S|$, $A$ is not
a hitting set for $\mathcal F$, and there exists an element $X \in \mathcal F$ such that 
$A \cap X =\emptyset$. To obtain a contradiction, note that $X$ is connected and thus there exists a path in $X$ from 
a vertex in $X\cap S \neq \emptyset$ to a vertex in $X \cap S_{i_a} \neq \emptyset$. This path does not contain any point of 
$A$ contradicting the choice of $A$ as an $(S, S_{i_a})$-separator.

We thus have a collection of $|S|$ vertex 
disjoint paths between $S$ and $S_{i_a}$ in $G$. Denote the unique path with endpoint $v \in S$ with $P_v$. Note that since the number of paths is $|S|$ and they are vertex disjoint, we have
 $S \cap P_v = \{v\}$.

Since the other end-point of $P_v$ is in $S_{i_a}$ and 
$S_{i_a} \cap C_a =\emptyset$, this in particular shows that the path $P_v$ intersects $G_a$ only at $v$.  

We now define the weak tree-decomposition $(T^a, \mathcal S^a)$ as follows. Let $T^a =T$, 
and for any $v \in S$, pick a node $t_v$ of $T$ with $t_v \in T_v$ (i.e., $v \in S_{t_v}$), 
and for any node $j$ of $T$ on the unique path from $i_a$ to 
$t_v$, add $v$ to $S_j$. By Proposition~\ref{prop:res}, this leads to a weak tree-decomposition $(T, \mathcal S')$ of $G$.
Define $(T^a, \mathcal S^a)$ as the restriction of $(T, \mathcal S')$ to $G^a$, which, 
again by Proposition~\ref{prop:res}, is a weak tree-decomposition of $G^a$.

Note that since $S'_{i_a} = S_{i_a} \cup S$ and $S_{i_a} \cap C_a =\emptyset$, we have $S^a_{i_a} = 
S'_{i_a} \cap V(G_a) = S$, and so Property $(i)$ holds. 
We now show that Property $(ii)$ above holds too, which finishes the proof of our theorem.

We first show that for any node $j$, $|S_j| \geq |S^a_j|$, or equivalently, $|S_j \cap S| \geq |S_j^a \cap S|$. 
Since the paths $P_v$ for $v\in S \cap S^a_j$ are vertex disjoint and intersect $V(G^a)$ only at 
$S \cap S^a_j$, in order to prove $|S_j \cap S| \geq |S_j^a \cap S|$, it will be enough to show that for any 
$v\in S \cap S^a_j$,  $S_j$ contains at least one vertex in $P_v$. 
Consider a vertex 
$v \in S_j^a \cap S$. By the definition of $S^a_j$, this means that either $v\in S_j$, or $j$ lies on the unique path 
between $i_a$ and $t_v$ in $T$. In the former case, we obviously have $\{v\} \subset S_j \cap P_v$. In the latter case, both the sets 
$S_{i_a}$ and $S_{t_v}$ have non-empty intersections with $P_v$ 
(which we recall is a path from $v \in S_{t_v}$ to a vertex in $S_{i_a}$). This means that $i_a$ and $t_v$ both belong to 
$\bigcup_{u\in P_v} T_u$, which is a connected subtree of $T$ by Lemma~\ref{lem:con}. This in particular means that the unique path between 
$i_a$ and $t_v$ is contained in $\bigcup_{u\in P_v} T_u$, in other words, there is a vertex $u$ on $P_v$ such that 
$j \in T_u$. Reformulating this, we get $u \in S_j \cap P_v$, which is what we wanted to prove.

To show $(ii)$, let now $j$ be a node of $T$ with $S^a_{j}$ of size strictly larger than $k$.
Since $|S| \leq k$, this means $S_j^a \setminus S \neq \emptyset$. Since $S^a_{j} \subset C_a \cup S$, this shows that 
$S^a_{j} \cap C_a \neq \emptyset$. By what we just proved, $|S_j|\geq |S^a_j|$, 
so $S_j$ is not a hitting set for $\mathcal F_a = \mathcal F \cup \{C_a\}$.  On the other hand, 
$S^a_j \subseteq S_j \cup S$, and $S^a_j$ has non-empty intersection with $C_a$, which shows that $S_j \cap C_a \neq \emptyset$. 
This means there exists $X \in \mathcal F$ such that $S_j \cap X = \emptyset$. 
 We show that $S^a_j \cap X = \emptyset$, which implies that $S^a_j$ is not a hitting set for $\mathcal F$. 
 
 Suppose this is not true, and so $S^a_j \cap X \neq \emptyset$. Since $S_j \cap X =\emptyset$, and $S_j^a \subset S_j \cup S$, 
 a point $v$ in 
 $S^a_j \cap X$ should belong to $S$. In other words, $j$ is on the unique path between $i_a$ and $t_v$ in $T$.
 To get a contradiction, note that $X$ intersects $S_{i_a}$ ($S_{i_a}$ is a hitting set for $\mathcal F$), 
 it intersects $S_{t_v}$ ($v\in X \cap S_{t_v}$), but it does not intersect $S_j$. This is in contradiction with Lemma~\ref{lem:con}.

 The proof of Theorem~\ref{thm:stbramwt} is now complete. 
\end{proof}

\section{Proof of the spectral lower bound on divisorial gonality}

Using the results of the previous section, we are now ready to give the proof of Theorems~\ref{thm:divgon} and~\ref{thm:random}.
 
 \medskip
 
The following is a direct corollary of Theorem~\ref{thm:stbramwt},  Proposition~\ref{prop:link}, and 
 Theorem~\ref{thm:divgontopbram}.
 \begin{cor}\label{cor1} Let $\Gamma$ be a metric graph. The divisorial gonality of $\Gamma$ is lower bounded by 
 the weak tree-width of any simple graph model $G=(V,E)$ of $\Gamma$.
 \end{cor}

 \begin{prop}\label{prop1} Let $G$ be a simple finite graph. We have $ 2wtw(G)\geq tw(G) + 1$.  
 \end{prop}

 \begin{proof} Let $(T, \mathcal S)$ be a weak tree-decomposition of $G$ of order $wtw(G)$, i.e., $|S_i| \leq wtw(G)$ 
 for any node $i$ of $T$. 
 We build a tree-decomposition $(T, Y)$ for $G$ of width at most $2wtw(G)-1$ out of $(T, \mathcal S)$. 
 
 Fix a node $\mathfrak r$ of $T$, and consider $T$ as being rooted at $\mathfrak r$. 
  Any node $i\neq \mathfrak r$ has a unique parent $p_i$ in the $\mathfrak r$-rooted tree $T$, which, we recall, 
  is the unique neighbor of $i$ in the unique path from $i$ to $\mathfrak r$ in $T$.  For any node $i$ of $T$ different from $\mathfrak r$, define $Y_i := X_i \cup X_{p_i}$.
 Furthermore, define $Y_\mathfrak r = X_\mathfrak r$. Let $\mathcal Y = \{Y_i\}_{i\in V(T)}$. 
 It is easy to check that $(T, \mathcal Y)$ is a tree-decomposition of $G$. In addition, we have $|Y_i| \leq 2 wtw(G)$, 
 from which the proposition follows.
\end{proof}

\subsection{Spectral lower bound for tree-width} We will need the following slight simplification of a
spectral lower bound for tree-width proved by Chandran-Subramanian~\cite{CS}.

Recall the definition of the discrete Laplacian $L_G$ on a connected graph $G=(V,E)$. For a function $f:V\to \mathbb \mathbb R$, $L_G(f)$ is the real-valued function on $V$ whose value at  a given vertex $v$ is given by
\[L_G(f)(v)=\sum_{u\sim v}f(v)-f(u),\]
where the sum is taken over all vertices $u$ adjacent to $v$. Denote by $\lambda_1(G)$ the smallest non-trivial eigenvalue of $L_G$.
\begin{thm}[Chandran-Subramanian~\cite{CS}]\label{prop:CS}
 For any connected graph $G=(V,E)$, the following holds
 \[tw(G)+1 \geq \frac{|V|\lambda_1(G)}{12\,d_{\max}},
 \]
 where $d_{\max}$ is the maximum valence of vertices of $G$.
\end{thm}

For the sake of completeness, we include the short proof of the above theorem. First recall the following variational characterization of $\lambda_1$:
\[\lambda_1 = \inf_{f: V \rightarrow \mathbb R \,\mathrm{with}\,\, \sum_v f(v)=0} \frac{\sum_{uv \in E} 
\bigl(f(u)-f(v)\bigr)^2}{\sum_{v\in V} f(v)^2}.\]
Let $Y$ and $Z$ be two disjoint non-empty subsets of $V$. Applying this to the (test) function $f$ defined by
 $f(z) = \frac 1{|Z|}$ for $z\in Z$, $f(y) = - \frac 1{|Y|}$ for $y \in Y$ and $f(w) = 0$ for
 any $w\in V \setminus (X \cup Y)$, which satisfies $\sum_{v\in V} f(v)=0$,  we get
\begin{equation}\label{eq*}
\lambda_1 \leq \Bigl(|E| - |E(Y)| - |E(Z)|\Bigr)\,\Bigl(\frac 1{|Y|} + \frac 1{|Z|}\Bigr),
\end{equation}
where $E(A)$ denotes the set of all edges with both endpoints in $A$. 
This is used repeatedly in the proof.
\begin{proof}[Proof of Theorem~\ref{prop:CS}] Denote by $n$ the number of vertices of $G$. 
For the sake of a contradiction, assume the inequality does not hold and 
let $(T, \mathcal S = \{S_i\})$ be a 
 tree decomposition of $G$ such that
 \[|S_i| \leq \frac{n\lambda_1}{12\,d_{\max}}.
 \]
 for any vertex $i$ of $T$. Denote by $\rho$ the quantity in the right hand side of the above equation. 
 The following argument, used also in the proof of Theorem~\ref{thm:stbramwt}, 
 shows the existence of a subset $S_i \in \mathcal S$ (thus, of size at most $\rho$) such that each component of $G\setminus S$ has size at 
 most $\frac {2(n-|S|)}3$. Suppose such a set $S$ does not exist. Consider a node $i$ of $T$.   For each neighbor $j$ of $i$ in the tree,  let $A_i(j)$ be the union of all the $S_k$ 
 with $k$ being a node in the subtree of $T \setminus \{i\}$ which contains $j$. 
 Our assumption implies one of the sets $A_i(j)\setminus S_{i}$, for $j$ adjacent to $i$ in $T$, 
 has size strictly larger than $\frac {2(n-|S_{i})|}3$. Give the orientation $i\rightarrow j$ to the edge 
 $e=\{i,j\}$ of $T$. Doing this for any node $i$, we give orientations to the edges of the tree exactly $|V(T)|$ times. 
  Since $T$ has $|V(T)|-1$ edges, at least one edge $\{i,j\}$ gets orientated twice, which means 
  $|A_i(j) \setminus S_i| > \frac{2(n-|S_i|)}3$ and $|A_j(i)\setminus S_j| > \frac{2(n-|S_j|)}3$. 
  Since $\rho\leq \frac n{12}$, this leads to a contradiction: indeed, the union of the 
  two disjoint sets $A_i(j) \setminus S_i$ and $A_j(i)\setminus S_j$ would have more than 
  $\frac{4n}3 - \frac{4\rho}3 >n$ vertices.

Let $S$ be a set of size at most $\rho$ such that all the connected components 
$Y_1,\dots, Y_s$ of $G \setminus S$ has size at most $\frac {2(n-|X|)}3$. Applying Inequality~\eqref{eq*} to 
the disjoint sets $Y_i$ and 
$Z_i=V \setminus (X\cup Y_i)$, and noting that  the quantity $|E| -|E(Y_i)| - |E(Z_i)|)$ is bounded by 
$|S|d_{\mathrm{max}}$, and $|Z_i| \geq \frac{n-|S|}3 \geq \frac{|Y|}{2}$, we get 
$\lambda_1 \leq \rho d_{\max} \frac{3}{|Y|} = \frac{n\lambda_1}{4|Y|}.$ This shows that each $Y_i$ has size at most 
$\frac n4$. 

Consider now the smallest $j$ such that 
$\overline Y_j= Y_1\cup \dots \cup Y_j$ has size at least $\frac n4$. Since each $Y_i$ has size at most $\frac n4$, 
$\overline Y_j$ has size at most $\frac n2$. Since $|S| \leq \rho \leq \frac{n}{12}$, we have $|\overline Z_j| \geq \frac n2 - \frac n{12} 
=\frac{5n}{12}$. Applying now Inequality~\eqref{eq*} to $\overline Y_j$ and 
$\overline Z_j=V\setminus (X \cup \overline Y_j)$, we 
get $\lambda_1 \leq \rho d_{\max}(\frac 4n+\frac {12}{5n}) = \frac{n}{12}(\frac 4n+\frac {12}{5n})\lambda_1$, which is a contradiction since 
$\lambda_1$ is strictly positive. 
 \end{proof}
The following result is a direct consequence of Theorem~\ref{thm:divgontopbram}, 
Theorem~\ref{thm:stbramwt}, Corollary~\ref{cor1} and Proposition~\ref{prop1}.
 \begin{thm}\label{thm:spectdisc} Let $\Gamma$ be a metric graph and $G=(V,E)$ a simple graph model of $\Gamma$. The divisorial gonality of $\Gamma$ satisfies the inequality
 $$\gamma_{\mathrm{div}}(\Gamma) \geq \frac{|V|\lambda_1(G)}{24 d_{\max}}.$$
\end{thm}

\subsection{Divisorial gonality of random graphs}\label{sec:random}
In this section we discuss some direct consequences of Theorem~\ref{thm:spectdisc} above when the underlying graph of the metric graph is a random graph (according to some model), and the edge lengths are arbitrary. 

Let $G=(V,E)$ be a simple graph, and $\Gamma$ any metric graph with $G$ as a model. The set of vertices of $G$ form a rank-determining set for $\Gamma$~\cite{Luo, HKN}, it follows that the divisor $\sum_{v\in V} (v)$ has rank at least one. In other words, $\gamma_{\mathrm{div}}(\Gamma) \leq n$. As we discuss now, in well-known classes of random graphs, we obtain that the 
divisorial gonality is $\Theta(n)$ with probability tending to one as $n$ goes to infinity. 

\medskip

Let  $G\in G(n,p)$ be  a Erdo\"s-R\'enyi random graph on $n$ vertices where any pair of two vertices are joined with an edge  independently with probability $p$. The threshold for the connectivity of $G$ is $\frac{\log(n)}n$. For $p>> \frac{\log(n)}n$, a random graph in $G(n,p)$ is with high probability connected, has, by Chernoff bound, maximum degree $d_{\max} = O(np)$,  and has $\lambda_1\sim pn$ by~\cite{Juhasz}. Thus, it follows from our results that $\gamma_{\mathrm{div}}(\Gamma) =\Theta(n)$ 

On the other hand, for $p < \frac{\log n}{n}$ the random graph $G\in G(n,p)$ is not necessarily connected. However, the threshold for the existence of a (unique) giant connected component in $G$ (i.e., of size linear in $n$) is $\frac 1n$. If in addition, we assume that $p>> \frac{1}{n}$, it follows from~\cite{Gao12} that the tree-width of a random graph in $G(n,p)$ is greater than $\beta n$ for some constant $\beta>0$, which in particular implies that a random Erd\"os-R\'enyi random graph with $p>> \frac 1n$ has divisorial gonality again $\Theta(n)$.  

\begin{cor}
The divisorial gonality of an Erd\"os-R\'enyi random graph in $G(n,p)$ is $\Theta(n)$ with probability tending to one as $n$ goes to infinity, provided that $pn >>1$.  
\end{cor}

It should be certainly possible to obtain sharper results. One might expect that when $pn$ is above a certain threshold, the divisorial gonality of a random graph in $G(n,p)$ is $(1-o(1))n$ with high probability.
  
  \medskip
  
Let $d\geq 3$ be an integer, and let $\epsilon>0$ be any small enough constant.  A random $d$-regular graph on $n$ vertices is  asymptotically almost surely connected, and by Friedman's theorem~\cite{Fri}, has $\lambda_1$ lower bounded by $d-2\sqrt{d-1}-\epsilon$. It follows that 
\begin{cor}
The divisorial gonality of a random $d$-regular graph is $\Theta(n)$ with probability tending to one as $n$ goes to infinity. 
\end{cor}
It should be again possible to obtain better bounds (and convergence theorems) for the divisorial gonality of a random $d$-regular graph as a function of the degree $d$.

\subsection{Proof of Theorem~\ref{thm:divgon}} 
In this final section, we present the proof of our main theorem. 

Let $\Gamma$ be a metric graph and let $G$ be a simple graph model of $\Gamma$ with $\ell_{\min} (G) = \ell_{\min}(\Gamma)$.

\medskip

Rescaling $\Gamma$ by a factor of $\beta=\frac 1{\ell_{\min}}$, we get a simple graph model $G'$ 
of $\Gamma' = \beta\Gamma$ where each edge of $G'$ has length at least one. 
Note that  $\lambda_1(\beta\Gamma) = 1/\beta^2 \lambda_1(\Gamma)$, while $\mu(\beta \Gamma) = \beta \mu(\Gamma)$ and 
$\ell_{\min}(\beta\Gamma) = \beta \ell_{\min}(\Gamma)$, so that the quantity $\lambda_1(.) \mu(.) \ell_{\min}(.)$ is scale free for a metric graph. 
The divisorial gonality of a metric graph is also easily seen to be scale free, which means in proving the inequality of 
Theorem~\ref{thm:divgon}, rescaling $\Gamma$ with a factor of $\beta$ if necessary, 
we can assume that $\ell_{\min}(\Gamma)=1$, and the simple graph model $G$ of $\Gamma$ has minimum edge length equal to one.

\medskip

We now subdivide the simple graph model $G=(V, E)$ of $\Gamma$ in the following way. For any edge 
$e=\{u,v\}$ of $G$ of length $\ell(e)$, let $u_1$ and $u_2$ be the two points of $\Gamma$ on $e$ at distance 
$\frac 1{16\deg_{G}(u)}$ and $\frac 1{16\deg_{G}(v)}$ from $u$ and $v$, respectively. 
Consider a set of points $A_e$ in the interval 
$[u_1, v_1]$ on the edge $e$, including $u_1$ and $v_1$, such that the distance between any two points of 
$A_e$ in the interval is at least $\frac 14$. Taking $A_e$ of maximum size, we see that 
$4\ell(e)-1\leq |A_e| \leq 4\ell(e)+2\leq 6 \ell(e)$. 

Let $\overline G=(\overline V, \overline E)$ be the subdivision of $G$ at all the points in the union of $A_e$, 
for $e$ an edge of $G$. We see that 

$$|\overline V| \geq \sum_{e\in E}|A_e| \geq \sum_{e\in E} 4 \ell(e) -1 \geq \sum_{e\in E}3\ell(e) = 3\mu(\Gamma).$$

(Note that we also have
$$|\overline V| = |V| + \sum_{e\in E}|A_e| \leq 2|E|+\sum_{e\in E}|A_e| = \sum_{e\in E} |A_e|+2 \leq \sum_{e\in E} 4\ell(e)+4 \leq \sum_{e\in E} 8\ell_e = 8 \mu (\Gamma),$$
which together give  $3\mu(\Gamma)\leq|\overline V|\leq 8 \mu(\Gamma)$.)

\medskip

We now claim that
\begin{claim}\label{claim}
 There is a constant $c_1$ such that $\lambda_1(\Gamma)\leq c_1 
\lambda_1(\overline G)$. 
\end{claim}
Here, $\lambda_1(\overline G)$ is the first non-trivial eigenvalue of the discrete Laplacian of $\overline G$. 
In the proof we will get $c_1 =128$, however we do not try to optimize the constant.

\medskip

Once this has been proved, applying Theorem~\ref{thm:spectdisc}, we get 
$$\gamma_{\div}(\Gamma) \geq \frac{|\overline V|\lambda_1(\overline G)}{24 d_{\max}} \geq \frac{3\mu(\Gamma)\lambda_1(\Gamma)}{24 c_1d_{\max}}.$$
Since we assume $\ell_{\min}(\Gamma) =1$, this leads to the proof of Theorem~\ref{thm:divgon} for the constant $C = \frac{1}{1024}.$

\medskip
We are thus left to prove the above claim.

\begin{proof}[Proof of Claim~\ref{claim}]
Recall that 
\[\lambda_1(\Gamma') = \inf_{f \in \mathrm{\Zh_0(\Gamma')}} \frac{\int_{\Gamma'} |f'|^2dx}{\int_{\Gamma'} f^2 dx}.\]

Recalling the variational characterization of $\lambda_1(\overline G)$, let 
$g: \overline V \rightarrow \mathbb R$ with $\sum_{v\in \overline V} g(v)=0$ and 
$$\lambda_1(\overline G) = \frac{\sum_{e=\{u,v\}\in \overline E}\,\, \bigl(g(u)-g(v)\bigr)^2}{\sum_{v\in \overline V}g(v)^2}.$$

For each vertex $v$ in $\overline V$ of degree $\deg(v)$, consider the disk $B(v)$ of radius 
$\frac{1}{16 \deg(v)}$ around $v$ in $\Gamma$. Note that $B(v)$ has volume $1/16$ for any vertex $v$ in $\overline G$. 

Note also that by the choice of $A_e$, $\Gamma \setminus \bigcup_{v\in \overline V}B(v)$ is a disjoint collection of 
segments of length at 
least $\frac 18$ (and at 
most $\frac 12$, by the maximality of each $A_e$). 

\medskip

Define the function $f:\Gamma \rightarrow \mathbb R$ as follows: first for any vertex $v\in \overline V$, 
define $f$ on the disk $B(v)$ to be the constant function taking value $g(v)$. Extend $f$ to whole of $\Gamma$ by linear interpolation on any segment of 
$\Gamma \setminus \bigcup_v B(v)$. Let $m = \frac 1{\mu(\Gamma)}\int_{\Gamma}f dx$ and consider the function 
$f-m$ which lies in $\mathrm{Zh}(\Gamma)$. We thus have 
\begin{equation}\label{eq0}
\lambda_1(\Gamma) \leq \frac{\int_{\Gamma} f'^2 dx}{\int_{\Gamma} (f-m)^2dx}.
\end{equation}

The function $f-m$ can be written as a sum $f_1+f_2$ where $f_1$ is the restriction of $f$ to $\bigcup_{v}B(v)$ extended by zero to whole of $\Gamma$, and $f_2 = f-m - f_1$.

\medskip

We have $\int_{\Gamma} f_1 dx= \sum_v \int_{B(v)} f_1 dx = \sum_v g(v) \mu(B(v)) = 1/16 \sum_v g(v) =0$,
and so $\int_{\Gamma} f_2 =0,$ as well.

In addition, since $f_2$ restricts to the constant function $-m$ on $\bigcup_v B(v)$, we have $\int_{\Gamma} f_1.f_2 =0$, 
which gives 

\begin{equation}\label{eq1}
 \int_{\Gamma} (f-m)^2 = \int_{\Gamma} f_1^2 + \int_{\Gamma} f_2^2 \geq \int_{\Gamma} f_1^2 = \sum_{v\in V} g(v)^2\mu(B(v)) = \frac 1{16} \sum_v g(v)^2.
\end{equation}

\medskip

We now  give an estimate of $\int_{\Gamma}f'^2$. Each connected component in $\Gamma\setminus \bigcup_v B(v)$ 
is a (unique) segment $I_e$
lying in the interior of an edge $e=\{u,v\}$ $\overline G$, and is adjacent to the two disks $B(u)$ and $B(v)$.

The function $f$ is linear affine of slope $\frac{g(u) - g(v)}{\ell(I)}$. Thus, we have

$$\int_{\Gamma} f'^2 =\sum_{e\in \overline E} \int_{I_e} f'^2 dx = 
\sum_{e=\{u,v\}\in \overline E} \frac{\bigl(g(u)-g(v)\bigr)^2}{\ell(I_e)}.$$

Given that the length of $I_e$ is at least $\frac 18$, we get 
\begin{equation}\label{eq2}
 \int_{\Gamma'}f'^2 dx \leq  8 \sum_{e=\{u,v\}\in \overline E} \bigl(g(u)-g(v)\bigr)^2.
\end{equation}

Equations~\eqref{eq0}, \eqref{eq1} and \eqref{eq2} together give
\[\lambda_1(\Gamma)\leq 128 \lambda_1(\overline G),\]
which is what we wanted to prove. 
\end{proof}

\begin{remark}\rm
 We refer to the paper of Cohen-Steiner and the first author~\cite{ACS} for a complement to Claim~\ref{claim}, and for an inequality in the other direction.
\end{remark}

\medskip

\thanks{{\it Acknowledgments}: O. A. likes to thank David Cohen-Steiner for his interest in the subject and for fruitful discussion and collaboration on related questions. Part of this research was conducted while O. A. and J. K. were visiting Max-Planck institute in mathematics in Bonn. They are grateful to the warm hospitality of the staff and excellent working conditions at the MPIM. }

\end{document}